\documentclass[12pt]{amsart}
\usepackage{amssymb}

\newtheorem{theorem}{Theorem}[section]
\newtheorem{proposition}[theorem]{Proposition}
\newtheorem{lemma}[theorem]{Lemma}
\newtheorem{remark}[theorem]{Remark}

\numberwithin{equation}{section}

\begin{document}

\baselineskip=15pt

\title[Equivariant bundles over wonderful compactifications]{On
equivariant principal bundles over wonderful compactifications}

\author[I. Biswas]{Indranil Biswas}

\address{School of Mathematics, Tata Institute of Fundamental
Research, Homi Bhabha Road, Mumbai 400005, India}

\email{indranil@math.tifr.res.in}

\author[S. S. Kannan]{S. Senthamarai Kannan}

\address{Chennai Mathematical Institute, H1, SIPCOT IT Park, Siruseri,
Kelambakkam 603103, India}

\email{kannan@cmi.ac.in}

\author[D. S. Nagaraj]{D. S. Nagaraj}

\address{The Institute of Mathematical Sciences, CIT
Campus, Taramani, Chennai 600113, India}

\email{dsn@imsc.res.in}

\subjclass[2000]{32Q26, 14M27, 14M17}

\keywords{Wonderful compactification, equivariant principal bundle, tangent 
bundle, stability}

\date{}

\begin{abstract}
Let $G$ be a simple algebraic group of adjoint type over $\mathbb C$, and
let $M$ be the
wonderful compactification of a symmetric space $G/H$.
Take a $\widetilde G$--equivariant principal $R$--bundle $E$ on
$M$, where $R$ is a complex reductive algebraic group and $\widetilde G$
is the universal cover of $G$. If the action
of the isotropy group $\widetilde H$ on the fiber of $E$ at the identity 
coset is irreducible, then we prove that $E$ is polystable with respect to any
polarization on $M$. Further, for
wonderful compactification of the quotient of $\text{PSL}(n,{\mathbb 
C})$, $n\,\neq\, 4$ (respectively, $\text{PSL}(2n,{\mathbb C})$) by the normalizer of the
projective orthogonal group (respectively,  the projective 
symplectic group), we prove that the tangent bundle is stable with
respect to any polarization on the wonderful compactification.
\end{abstract}

\maketitle

\section{Introduction}\label{se1}

Let $G$ be a semi-simple linear algebraic group of adjoint type defined over
the field $\mathbb C$ of complex numbers.
The universal cover of $G$ will be denoted by $\widetilde G$. Let $\sigma$
be an algebraic involution of $G$ induced by an automorphism
$\widetilde{\sigma}$ of $\widetilde G$ of order two. The fixed point
subgroup of $G$ for $\sigma$ will be
denoted by $H.$ The quotient $G/H$ is an affine variety. De Concini and Procesi constructed
a compactification of $G/H$ which is known as the wonderful compactification \cite{DP}.
The left--translation action of $G$ on $G/H$ extends to an action of $G$ on
the wonderful compactification $\overline{G/H}$. This produces an action of
$\widetilde G$ on $\overline{G/H}$.
Our aim here is to investigate the $\widetilde G$--equivariant principal
bundles on $\overline{G/H}.$

Let $R$ be a connected reductive complex linear algebraic group. Let $E\,\longrightarrow\,
\overline{G/H}$ be a $\widetilde G$--equivariant algebraic principal
$R$--bundle. The inverse image of $H$ in $\widetilde G$ will be denoted
by $\widetilde H$. Since the isotropy for the point
$e\,=\, eH\, \in\, G/H$ for the action of $\widetilde G$ is $\widetilde H$, we 
have an action of $\widetilde H$ on the fiber $E_e$. Let
\begin{equation}\label{e2}
\gamma\, :\, {\widetilde H}\, \longrightarrow\, \text{Aut}^R(E_e)
\end{equation}
be the corresponding homomorphism, where $\text{Aut}^R(E_e)$ is the group of
algebraic automorphisms of $E_e$ that commute with the action of $R$ on it.
The groups $\text{Aut}^R(E_e)$ and $R$ are isomorphic by an isomorphism which is 
unique up to inner automorphisms.

We prove the following (see Proposition \ref{prop1}):

\textit{If $\gamma({\widetilde H})$ is not contained in any proper parabolic subgroup of
${\rm Aut}^R(E_e),$ then the principal $R$--bundle $E$ is polystable with respect to
every polarization of $\overline{G/H}$.}

For equivariant vector bundles on $\overline{G/H}$, the above proposition can be 
improved; see Proposition \ref{prop2} for the precise statement.

In Section \ref{sec3}, we consider the following two symmetric spaces
$\text{PSL}(n, {\mathbb C})/{\text{NPSO}(n, {\mathbb C})}, \,\, n\neq 4$
and $\text{PSL}(2m, {\mathbb C})/{\text{NPSp}(2m, {\mathbb C})}, \,\, m\geq 2, $ where $\text{NPSO}(n, {\mathbb C})$ (respectively, 
$\text{PSp}(n, {\mathbb C})$ denote the normalizer of the projective 
orthogonal group (respectively, projective symplectic group) in 
$\text{PSL}(n, {\mathbb C})$ (respectively, $\text{PSL}(2m, {\mathbb C})).$
See \cite[p. 7, Lemma]{DP}, for details.

The first one of the above two symmetric spaces corresponds to the involution 
$\sigma$ of $\text{PSL}(n, {\mathbb C})$ induced by the automorphism
$$
A\, \longmapsto\, (A^t)^{-1}\, 
$$
of $\text{SL}(n, {\mathbb C})$. The second one corresponds to the involution
$\sigma$ of $\text{PSL}(2m, {\mathbb C})$ induced by the automorphism
$$
A\, \longmapsto\, J^{-1}(A^t)^{-1}J \, 
$$
of $\text{SL}(2m, {\mathbb C}),$ where 
$$
J\, :=\, \begin{pmatrix}
0 & I_{m\times m}\\
-I_{m\times m} & 0
\end{pmatrix}.
$$

In Theorem \ref{thm1}, and Remark \ref{rem3} we prove the following:

\textit{For the wonderful compactification $ \overline{G/H}$ of the
above two symmetric spaces, the tangent bundle $T \overline{G/H}$ is 
stable with respect to every polarization of $\overline{G/H}$.}

As pointed out by the referee, Theorem \ref{thm1} remains valid as long as the three
conditions stated in the beginning of Section \ref{se3} are valid and
the $H$--module ${\rm Lie}(G)/{\rm Lie}(H)$ is irreducible (see Remark \ref{rem1}). 

\section{Polystability of irreducible equivariant bundles}

We continue with the above notation. The
wonderful compactification $\overline{G/H}$ of the quotient $G/H$ will be
denoted by $M$. The left--translation action of $G$ on $G/H$ extends to an action
$G\times M\, \longrightarrow\, M$. Using the natural projection ${\widetilde G}\,
\longrightarrow\, G$ from the universal cover, the above action of $G$ on $M$
produces an action
\begin{equation}\label{e1}
\rho\, :\, {\widetilde G}\times M\, \longrightarrow\, M
\end{equation}
of $\widetilde G$ on $M$.

Let $R$ be a connected reductive complex linear algebraic group. An \text{equivariant}
principal $R$--bundle on $M$ is an algebraic principal $R$--bundle on $M$ equipped with
a lift of the left-action of $\widetilde G$ in \eqref{e1}. More precisely, an equivariant
$R$--bundle is a pair $(E\, ,\widetilde{\rho})$, where $E\,\longrightarrow\, M$ is an
algebraic principal $R$--bundle, and
$$
\widetilde{\rho}\, :\, {\widetilde G}\times E\,\longrightarrow\, E
$$
is an algebraic action of $\widetilde G$ on the total space of $E$, such that the
following two conditions hold:
\begin{enumerate}
\item the projection of $E$ to $M$ intertwines the actions of
$\widetilde G$ on $E$ and $M$, and

\item the action of $\widetilde G$ on $E$ commutes with the action of $R$ on $E$.
\end{enumerate}

Let $(E\, ,\widetilde{\rho})$ be an equivariant principal $R$--bundle
on $M$. Let
$$
\text{Ad}(E)\, :=\, E\times^R R\,\longrightarrow\, M
$$
be the fiber bundle associated to $E$ for the conjugation action of $R$ on
itself. Since the conjugation action of $R$ on itself preserves the group
structure of $R$, the fibers of $\text{Ad}(E)$ are groups isomorphic to
$R$. To see an explicit isomorphism of $R$ with a fiber $\text{Ad}(E)_x$,
where $x\in\, X$, fix a point $z_0\, \in\, E_x$. Now the map
\begin{equation}\label{h1}
R\, \longrightarrow\, \text{Ad}(E)_x
\end{equation}
that sends any $g\, \in\, R$ to the
equivalence class of $(z_0\, ,g)\, \in\, E\times R$ is an isomorphism of
groups. Therefore, $\text{Ad}(E)_x$ is identified
with $R$ by an isomorphism which is unique up to an inner automorphism of $R$.

The equivariant $R$--bundle $(E\, ,\widetilde{\rho})$ is called
\textit{irreducible} if the image $\gamma({\widetilde H})$ is not contained in
any proper parabolic subgroup of $\text{Aut}^R(E_e)\,=\, \text{Ad}(E)_e$, where 
${\widetilde H}\, \subset\, \widetilde{G}$, as before, is the inverse image of 
$H$, and $\gamma$ is the homomorphism in \eqref{e2}.

Fix a polarization ${\mathcal L}\,\in\, H^2(M,\, {\mathbb Q})$ on $M$,
meaning ${\mathcal L}$ is the class of a very ample line bundle on $M$. The
degree of a torsionfree coherent sheaf on $M$ is defined using $\mathcal L$
as follows: for a torsionfree coherent sheaf $F$ on $M,$ 
$$\text{degree}\, (F)\,:=\, (c_1(F)\cdot {\mathcal L}^{n-1})\cap
[M]\, \in\, \mathbb Z\, , $$
where $n$ is the (complex) dimension of $M$.

An algebraic vector bundle $V$ on $M$ is called \textit{semistable}
(respectively, \textit{stable}) if for every coherent subsheaf
$F\, \subset\, V$ with $0 \,< \,\text{rank}(F) <
\text{rank}(V)$, the inequality 
$$
\mu(F) \,:= \,\frac{\text{degree}(F)}{\text{rank}(F)} \,\leq\,
\mu(V) \,:= \,\frac{\text{degree}(V)}{\text{rank}(V)}
$$
(respectively, $\mu(F) \,<\, \mu(V)$) holds. A semistable vector bundle
is called \textit{polystable} if it is a direct sum of stable vector bundles.

A principal $R$--bundle $E$ on $M$ is called \textit{semistable}
(respectively, \textit{stable}) if for every maximal proper parabolic 
subgroup $P\,\subset\, R$ and for every reduction
$\tau \,:\, U \,\longrightarrow (E\vert_U)/P$ over some
Zariski open set $U$ of $M$ such that the codimension of
$M\setminus U$ is at least two,
we have $\text{degree}(\tau^*T_{\rm rel}) \,\geq\, 0$ 
(respectively, $\text{degree}(\tau^*T_{\rm rel}) \,> \,0$), where $T_{rel}$
is the relative tangent bundle for the natural projection of
$(E/P)\vert_U$ to $U$. A principal $R$--bundle $E$ on $M$ is
said to be \textit{polystable} if there is a parabolic subgroup $P$ of $R$ 
and a reduction of structure group $E_L\,\subset\, E$ to a Levi
factor $L\, \subset\, P$ such that 
\begin{enumerate}
\item the principal $L$ bundle $E_L$ on $M$ is stable, and

\item the principal $P$--bundle obtained by extending the structure group of
$E_L$
$$E_P\,:=\,E_L\times^L P$$ has the 
property
that for any character $\chi$ of $P$ which is trivial on the center of
$R$, the line bundle on $M$ associated to $E_P$ for $\chi$ has degree zero.
\end{enumerate}

\begin{proposition}\label{prop1}
Let $(E\, ,\widetilde{\rho})$ be an irreducible equivariant $R$--bundle on $M$.
Then the principal $R$--bundle $E$ is polystable.
\end{proposition}

\begin{proof}
We will first prove that $E$ is semistable. Assume that $E$ is not semistable. Then,
there is a Zariski open subset $U\, \subset\, M$ such that the complement
$M\setminus U$ is of codimension at least
two, a proper parabolic subgroup $P\, \subset\, R$, and an algebraic
reduction of structure group
$$
E_P\, \subset\, E\vert_U
$$
of $E$ to $P$ over $U,$ such that $E_P$ is the Harder-Narasimhan
reduction for $E$ (see \cite{AAB} for Harder-Narasimhan reduction). Let
$$
\text{Ad}(E_P)\, :=\, E_P\times^P P\,\longrightarrow\, U
$$
be the adjoint bundle associated to $E_P$ for the conjugation action of
$P$ on itself. Just as for $\text{Ad}(E)$, the fibers of $\text{Ad}(E_P)$
are groups isomorphic to $P$ because the conjugation action of
$P$ on itself preserves the group structure.
The natural inclusion of $E_P\times P$ in $(E\vert_U)\times G$
produces an inclusion
$$
\text{Ad}(E_P)\,\hookrightarrow\, \text{Ad}(E)\vert_U\, .
$$
In the isomorphism in \eqref{h1} if we take $z_0\,\in\, (E_P)_x$, then
the isomorphism sends $P$ isomorphically to the fiber $\text{Ad}(E_P)_x$.
Therefore, $\text{Ad}(E_P)_x$ is a parabolic subgroup of
$\text{Ad}(E)_x$ because $P$ is a parabolic subgroup of $R$.

Take an element $g\,\in\, \widetilde G$ such that the point
$$
\overline{g}\, :=\, \rho(g\, ,e)\, \in\, G/H\, \subset\, M
$$
lies in the above open subset $U\,\subset\, M$. Consider the automorphism of $E$
defined by $z\,\longmapsto\, \widetilde{\rho}(g\, ,z)$. The subgroup
$g{\widetilde H} g^{-1}\, \subset\,\widetilde G$ preserves the fiber $E_{\overline{g}}$
because $g {\widetilde H} g^{-1}$ is the isotropy of $\overline{g}$ for the
action $\rho$ in \eqref{e1}. Therefore, we get a homomorphism
\begin{equation}\label{e3}
\theta\, :\, g {\widetilde H} g^{-1}\, \longrightarrow\, \text{Ad}(E)_{\overline{g}}\, .
\end{equation}
The action of $g$ on $E$ produces an isomorphism of algebraic groups
\begin{equation}\label{eta}
\eta\, :\, \text{Ad}(E)_e\, \longrightarrow\, \text{Ad}(E)_{\overline{g}}\, .
\end{equation}
The following diagram is commutative
\begin{equation}\label{e4}
\begin{matrix}
{\widetilde H} & \stackrel{\gamma}{\longrightarrow} &\text{Ad}(E)_e\\
~\Big\downarrow g' && ~ \Big\downarrow \eta\\
g{\widetilde H}g^{-1}& \stackrel{\theta}{\longrightarrow} & \text{Ad}(E)_{\overline{g}}
\end{matrix}
\end{equation}
where $\gamma$, $\theta$ and $\eta$ are defined in \eqref{e2}, \eqref{e3} and
\eqref{eta} respectively, and $g'(y)\,=\, gyg^{-1}$.

The polarization ${\mathcal L}$ on $M$ is preserved
by the action of $\widetilde G$ in \eqref{e1} because $\widetilde G$ is connected.
The action of $\widetilde G$ on $E$ preserves the pair $U$ and
$E_P$ because the Harder--Narasimhan reduction is unique for a given
polarization. Consequently, the image $\theta(g {\widetilde H} g^{-1})$
in \eqref{e3} is contained in the parabolic subgroup
\begin{equation}\label{z1}
\text{Ad}(E_P)_{\overline{g}}\,\subset\, \text{Ad}(E)_{\overline{g}}\, .
\end{equation}
On the other hand, since $\gamma({\widetilde H})$ is not contained in
any proper parabolic subgroup of
$\text{Ad}(E)_e$, from the commutativity of the diagram in \eqref{e4} we conclude that
$\theta(g {\widetilde H}g^{-1})$ is not contained in any proper parabolic subgroup of
$\text{Ad}(E)_{\overline{g}}$. But this is in contradiction with \eqref{z1}.
Therefore, we conclude that the principal $R$--bundle $E$ is semistable.

We will now prove that $E$ is polystable. Note that $E$ is polystable 
if and only if the adjoint vector bundle $\text{ad(E)} \,:=\,E\times^R \mathfrak{R}$
is polystable, where $\mathfrak{R}$ is the Lie algebra of the reductive group
$R$ (see \cite[p. 224, Corollary 3.8]{AB}). Thus it is enough to prove
that $\text{ad}(E)$ is polystable. To prove that $\text{ad(E)}$ is
polystable, we simply replace the Harder--Narasimhan reduction
in the above proof by the socle reduction of the semistable vector bundle
$\text{ad(E)}$ (see \cite[p. 218, Proposition 2.12]{AB}). Since the socle reduction
is unique, simply repeating the above proof we get that $\text{ad}(E)$
is polystable.
\end{proof}

\section{Equivariant vector bundles}\label{se3}

Let $G$ and $H$ be as before. From now on we will assume the following:

\begin{enumerate}
\item The connected component $\widetilde{H}^0\,\subset\, \widetilde{H}$ containing
the identity element is a simple algebraic group.

\item For any maximal torus $T_{0}$ of the
connected component $\widetilde{H}^{0}\,\subset\, \widetilde{H}$
containing the identity element, for any Borel 
subgroup $B_{0}$ of $\widetilde{H}^{0}$ containing $T_{0}$, for any Borel subgroup 
$B$ of $\widetilde{G}$ containing $B_{0}$ and for any maximal torus $T$ of $B$ containing 
$T_{0}$, the restriction map $X(B)^{+}\longrightarrow X(B_{0})^{+}$ is surjective. Here
$X(B)$ (respectively, $X(B)^+$) denotes the set of all
characters (respectively, dominant characters) of $B$. Similarly 
$X(B_0)^+$ is defined.

\item The restriction of any simple root $\alpha$ of $\widetilde{G}$ (with 
respect to  $T$ and $B$) to $B_{0}$ is nonzero and it is a 
nonnegative integral linear combination of simple roots of $\widetilde{H}^{0}$
(with respect to $T_{0}$ and $B_{0}$).
\end{enumerate}

\begin{remark}\label{rem0}
{\rm The pairs of groups $G \,=\, {\rm PSL}(n,\mathbb{C})\, ,\, H \,=\,
{\rm NPSO}(n,\mathbb{C})\, ,
\,\, n\neq 2, 4$  and $G \,=\, {\rm PSL}(2m,\mathbb{C})\, ,\, H \,=\,
{\rm PSp}(2m,\mathbb{C})\, , \,\,
 m\geq 2,$ satisfy the  above conditions (1), (2), (3). To see this we consider
${\rm SO}(n,\mathbb{C})$ (respectively, ${\rm Sp}(2m,\mathbb{C})$) as the subgroup
of the special linear group preserving the nondegenerate symmetric (respectively,
skew-symmetric) bilinear form 
$$
\sum_{i=1}^n X_iY_{n+1-i}\, \ 
({\rm respectively,}\ \sum_{i=1}^m (X_iY_{2m+1-i}-X_{m+i}Y_{m+1-i}) )\, .
$$
 \begin{enumerate}
 
 \item The simplicity of the above groups $H$ follows from the facts
 about the classical groups of type $A, B, C$ and $D$.
 
 \item By choice of the nondegenerate bilinear forms we see that the
 inclusion of maximal torus $T_0\,\subset\, T$ and the inclusion of 
 Borel subgroup $B_0\,\subset\, B$ satisfies the hypothesis (2), (3) above
 (see \cite[p. 215, p. 243, p. 272]{FH}).
 \end{enumerate}}
\end{remark}

\begin{lemma}\label{res}
Every irreducible representation of $\widetilde{H}^{0}$ is a restriction of 
some representation of $\widetilde{H}$.
\end{lemma}

\begin{proof}
Let $V$ be an irreducible representation of $\widetilde{H}^0$.
If $V$ is the trivial representation, then there is nothing to prove. 
Otherwise, let $\lambda$ be the highest weight of $V$. Then, by using the
part (2) of the hypothesis, there is a dominant character $\chi$ of $B$ whose 
restriction to $B_{0}$ is $\lambda$. Hence, the irreducible representation 
$V(\chi)$ of $\widetilde{G}$ with highest weight $\chi$ is a direct sum of 
$V$ with multiplicity one and of some irreducible representations of 
$\widetilde{H}^0$ with highest weights $\mu$ satisfying $\mu\,<\,\lambda$ for 
the dominant ordering in $\widetilde{H}^0$. This is because every weight 
$\nu$ of $V(\chi)$ satisfies $\nu\,<\,\chi$ for the dominant ordering in 
$\widetilde{G}$ with respect to $T$ and $B$ and by using the part (3) of the 
hypothesis that the restriction to $B_{0}$ of every simple root $\alpha$ of 
$\widetilde{G}$ with respect to $T$ and $B$ is nonzero and is a 
nonnegative integral linear combination of simple roots of $\widetilde{H}^0$.

Since any two Borel subgroups of $\widetilde{H}^0$ are conjugate in 
$\widetilde{H}^0$ and any 
two maximal tori of $B_{0}$ are conjugate in $B_{0}$, we may choose the
representatives of $\widetilde{H}/\widetilde{H}^0$ to lie in both 
$N_{\widetilde{H}}(B_{0})$ and $N_{\widetilde{H}}(T_{0})$. Consequently, the 
finite group $\widetilde{H}/\widetilde{H}^{0}$ acts on the group of characters 
of $B_{0}$ preserving the dominant characters (not necessarily preserving
pointwise). Further, since the representatives of $\widetilde{H}/\widetilde{H}^{0}$ can be
chosen in $N_{\widetilde{H}}(B_{0})$ the action of $\widetilde{H}/\widetilde{H}^{0}$
preserves the positive roots of $\widetilde{H}^{0}$ with respect to $B_0$. Thus, the
$\widetilde{H}$--span of $V$ in $V(\chi)$ 
is a direct sum of irreducible representations of 
$\widetilde{H}^0$ whose highest weights are of the form $\sigma(\lambda)$
with $\sigma$ running over the elements of the finite 
group $\widetilde{H}/\widetilde{H}^0$. By the previous paragraph, for every 
$\sigma \,\in\, \widetilde{H}/\widetilde{H}^{0}$, either 
$\sigma(\lambda)\,<\, \lambda$ or $\sigma(\lambda)\,=\,\lambda$. 
On the other hand, if for some 
$\sigma \,\in\, \widetilde{H}/\widetilde{H}^{0}$ we have
$\sigma(\lambda)\,<\, \lambda$ then 
$$\lambda \,=\, \sigma^n(\lambda)\,<\,
\sigma^{n-1}(\lambda) \,< \,\ldots \, \sigma(\lambda)\,<\,\lambda\, ,$$
where $n$ is the order of $\sigma$, which is a contradiction. 
Therefore, the $\widetilde{H}$--span of $V$ coincides with $V$, implying that 
$V$ is a restriction of a representation of $\widetilde{H}$.
\end{proof}

A vector bundle $W$ of rank $r$ on $M$ is called
{\it equivariant} if $W$ corresponds to an equivariant
principal $\text{GL}(r,{\mathbb C})$--bundle. Equivalently, an
equivariant vector bundle is a pair $(W\, , \widetilde{\rho}),$ where
$W$ is an algebraic vector bundle on $M,$ and
$$
\widetilde{\rho}\, :\, {\widetilde G}\times W\,\longrightarrow\, W
$$
is an algebraic action of $\widetilde G$ on $W,$ such that the
following two conditions hold:
\begin{enumerate}
\item the projection of $W$ to $M$ intertwines the actions of
$\widetilde G$ on $W$ and $M,$ and

\item the action $\widetilde{\rho}$ preserves the linear structure
of the fibers of $W$.
\end{enumerate}
An equivariant vector bundle $(W\, , \widetilde{\rho})$ is called
\textit{irreducible} if the representation 
$$
\rho_e \,:\, {\widetilde H}\, \longrightarrow\,\text{GL}(W_e)
$$
given by the action of the isotropy subgroup for the point
$e\, \in\, M$ is irreducible. Note that the
irreducible equivariant vector bundles of rank $r$ correspond to the
irreducible equivariant principal $\text{GL}(r,{\mathbb C})$--bundles.

\begin{proposition}\label{prop2}
Let $(W\, , \widetilde{\rho})$ be an irreducible equivariant vector bundle on $M$
of rank $r$. Then either $W$ is stable, or $W$ admits a decomposition
$$
W\,=\, L^{\oplus r}\, ,
$$
where $L$ is a line bundle on $M$.
\end{proposition}

\begin{proof}
The vector bundle $W$ is polystable by Proposition \ref{prop1}. Therefore,
$W$ can be uniquely decomposed as
\begin{equation}\label{f3}
W\,=\, \bigoplus_{i=1}^\ell W_i \bigotimes\nolimits_{\mathbb C}
H^0(M,\, W\bigotimes W^\vee_i)\, ,
\end{equation}
where $W_i$ are distinct stable vector bundles on $M$. The above assertion of
uniqueness means the following: if
$$
W\,=\, \bigoplus_{j=1}^{\ell'} W'_j\bigotimes\nolimits_{\mathbb C} {\mathbb C}^{r_j}\, ,
$$
where $W'_1\, ,\cdots\, , W'_{\ell'}$ are non-isomorphic stable vector bundles,
then $\ell\,=\, \ell'$ and there is a permutation $\alpha$ of
$\{1\, ,\cdots\, ,\ell\}$ such that the subbundle
$W_i \bigotimes_{\mathbb C} H^0(M,\, W\bigotimes W^\vee_i)$ of $W$
in \eqref{f3} coincides with the above subbundle
$$
W'_{\alpha(i)}\bigotimes\nolimits_{\mathbb C} {\mathbb C}^{r_{\alpha(i)}}\, \subset\, W\, .
$$
In particular, $W_i$ is isomorphic to $W'_{\alpha(i)}$ and $\dim
H^0(M,\, W\bigotimes W^\vee_i)\,=\, r_{\alpha(i)}$. This uniqueness follows
immediately from the Krull--Schmidt decomposition of a vector bundle (see
\cite[p. 315, Theorem 3]{At}) and the fact that for any two non-isomorphic
stable vector bundles $W^1$ and $W^2,$
$$
H^0(M,\, W^1\bigotimes (W^2)^\vee)\,=\, 0\, .
$$
For any $g\, \in\, \widetilde{G}$, let
\begin{equation}\label{rg}
{\rho}_g\, :\, M\, \longrightarrow\, M
\end{equation}
be the automorphism defined by $x\,\longmapsto\, \rho(g^{-1}\, ,x)$.
Similarly, let
\begin{equation}\label{rg2}
\widetilde{\rho}_g\, :\, W\, \longrightarrow\, W
\end{equation}
be the map defined by $v\, \longmapsto\,\widetilde{\rho}(g\, ,v)$; note that
$\widetilde{\rho}_g$ is an isomorphism of the vector bundle $W$ with the
pullback $\rho^*_g W$. The pullback
$$
\rho^*_g W\,=\, \bigoplus_{i=1}^\ell
\rho^*_g (W_i \bigotimes\nolimits_{\mathbb C} H^0(M,\, W\bigotimes W^\vee_i))
$$
of the decomposition in \eqref{f3} coincides with the unique decomposition
(unique in the above sense) of $\rho^*_g W$. Hence the isomorphism
$\widetilde{\rho}_{g}$ takes the above decomposition of $\rho^*_g W$
to a permutation $\nu(g)$ of the decomposition of $W$ in \eqref{f3}.
Therefore, we get a map
$$
\nu\, :\, \widetilde{G}\, \longrightarrow\, P(\ell)\, ,
$$
where $P(\ell)$ is the group of permutations of $\{1\, ,\cdots\, ,\ell\}$,
that sends any $g\,\in\, \widetilde{G}$ to the above permutation $\nu(g)$.
This map $\nu$ is clearly continuous, the permutation $\nu(e)$ is the identity
map of $\{1\, ,\cdots\, ,\ell\},$ and $\widetilde G$ is connected. These
together imply that $\nu$ is the constant map to the
identity map of $\{1\, ,\cdots\, ,\ell\}$. In other words, the action of
$\widetilde G$ on $W$ preserves the subbundle
$$
W_i \bigotimes\nolimits_{\mathbb C} H^0(M,\, W\bigotimes W^\vee_i)\,\subset\, W
$$
in \eqref{f3} for every $i$.

We will next show that each vector bundle $W_i$ admits a 
$\widetilde G$ equivariant structure.

We have noted above that the action of $\widetilde G$ on $W$ preserves the
subbundle $W_i \bigotimes_{\mathbb C} H^0(M,\, W\bigotimes W^\vee_i)$. The automorphism
$\widetilde{\rho}_g$ of $W_i \bigotimes_{\mathbb C}H^0(M,\, W\bigotimes W^\vee_i)$
(see \eqref{rg2}) produces an isomorphism
$$
({\rho}^*_g W_i) \bigotimes\nolimits_{\mathbb C}H^0(M,\,
W\bigotimes W^\vee_i) \,\stackrel{\sim}{\longrightarrow}\,
W_i \bigotimes\nolimits_{\mathbb C}H^0(M,\, W\bigotimes W^\vee_i)\, ,
$$
where $\rho_g$ is defined in \eqref{rg}.
Since the vector bundle $W_i$ is indecomposable, the vector bundle
${\rho}^*_g W_i$ is also indecomposable, and hence from
\cite[p. 315, Theorem 2]{At} we know that $W_i$ is isomorphic to ${\rho}^*_g W_i$.

Take any integer $1\,\leq\, i\,\leq\, \ell$.
Let $\text{Aut}(W_i)$ denote the group of all algebraic automorphisms of the vector bundle
$W_i$. It should be clarified that any automorphism of $W_i$ lying in $\text{Aut}(W_i)$
is over the identity map of $M$. The group $\text{Aut}(W_i)$ is the Zariski open subset
of the affine space $H^0(M,\, W_i\bigotimes W^\vee_i)$ defined by the locus of
invertible endomorphisms. Therefore, $\text{Aut}(W_i)$ is a connected complex
algebraic group.

Let $\widetilde{\text{Aut}}(W_i)$ denote the set of all pairs of the form
$(g\, ,f)$, where $g\, \in\,\widetilde G$ and
$$
f\, :\, \rho_g^* W_i\,\longrightarrow\, W_i
$$
is an algebraic isomorphism of vector bundles, where $\rho_g$ is the automorphism
in \eqref{rg}. This set $\widetilde{\text{Aut}}(W_i)$
has a tautological structure of a group
$$
(g_2\, ,f_2)\cdot (g_1\, ,f_1)\,=\, (g_1g_2\, , f_2\circ \rho_{g_2}^*(f_1))\, .
$$
Since the vector bundle $W_i$ is simple (recall that it is stable), the group
$\widetilde{\text{Aut}}(W_i)$ fits in a short exact sequence of algebraic groups
\begin{equation}\label{h2}
e\,\longrightarrow\,{\mathbb C}^*\,\longrightarrow\,\widetilde{\text{Aut}}(W_i) \,
\stackrel{\delta_i}{\longrightarrow}\, \widetilde G\,
\longrightarrow\,e\, ,
\end{equation}
where $\delta_i$ sends any $(g\, ,f)$ to $g$. Note that the earlier 
observation that $W_i$ is isomorphic to ${\rho}^*_g W_i$ for all $g\,\in\, 
\widetilde G$ implies that the homomorphism $\delta_i$ in \eqref{h2} is 
surjective.

The Lie algebras of $\widetilde G$ and
$\widetilde{\text{Aut}}(W_i)$ will be denoted by $\mathfrak g$ and $A(W_i)$ respectively. Let
\begin{equation}\label{h3}
\delta'_i\, :\, A(W_i)\,\longrightarrow\,{\mathfrak g}
\end{equation}
be the homomorphism of Lie algebras corresponding to
$\delta_i$ in \eqref{h2}. Since ${\mathfrak g}$ is semisimple, there is a
homomorphism of Lie algebras
$$
\tau_i\, :\, {\mathfrak g}\,\longrightarrow\,A(W_i)
$$
such that
\begin{equation}\label{es}
\delta'_i\circ\tau_i\,=\, \text{Id}_{\mathfrak g}
\end{equation}
\cite[p. 91, Corollaire 3]{Bo}. Fix a homomorphism $\tau_i\, :\, {\mathfrak g}\,
\longrightarrow\,A(W_i)$ satisfying \eqref{es}. Since the group $\widetilde G$
is simply connected, there is a unique algebraic representation
$$
\widetilde{\tau}_i\, :\, {\widetilde G}\,\longrightarrow\,
\widetilde{\text{Aut}}(W_i)
$$
such that the corresponding homomorphism of Lie algebras coincides with
$\tau_i$. From \eqref{es} it follows immediately that $\delta_i\circ\widetilde{\tau}_i
\,=\, {\rm Id}_{\widetilde G}$.

We now note that $\widetilde{\tau}_i$ defines an action of $\widetilde G$ 
on $W_i$. The pair $(W_i\, ,\widetilde{\tau}_i)$ is an equivariant vector 
bundle. In particular, the fiber $(W_i)_e$ is a representation of 
$\widetilde H$.

Consider the decomposition of $W$ in \eqref{f3}. The actions of $\widetilde 
G$ on $W$ and $W_i$ together define a linear action of $\widetilde G$ on 
$H^0(M,\, W\bigotimes W^\vee_i)$. With respect to these actions, the 
isomorphism in \eqref{f3} is $\widetilde G$--equivariant.

Since the isomorphism in \eqref{f3} is $\widetilde G$--equivariant,
we get an isomorphism of representations of $\widetilde H$
$$
W_e\, =\, \bigoplus_{i=1}^\ell (W_i)_e\bigotimes\nolimits_{\mathbb C}
H^0(M,\, W\bigotimes W^\vee_i)\, ;
$$
we noted above that both $(W_i)_e$ and $H^0(M,\, W\bigotimes W^\vee_i)$ are
representations of $\widetilde H$.
Since the $\widetilde H$--module $W_e$ is irreducible, we conclude that
$\ell\,=\, 1$. So
$$
W\,=\, W_1\bigotimes H^0(M,\, W\bigotimes W^\vee_1)\, ,
$$
and we have an isomorphism of representations of $\widetilde{H}$
$$
W_e\, =\, (W_1)_e\bigotimes H^0(M,\, W\bigotimes W^\vee_1)\, .
$$

As in Section \ref{se1}, let $\widetilde{\sigma}$ be the lift of
$\sigma$ to $\widetilde{G}$.

By Lemma \ref{res}, any irreducible $\widetilde{H}^{0}$ module
$V$ of $W_{e}$ is a restriction of a $\widetilde{H}$
module. Hence, it follows that the irreducible $\widetilde{H}$
module $W_{e}$ is an irreducible $\widetilde{H}^{0}$ module as well.

Recall the assumption that $H^{0}$ is a simple algebraic group. From the 
irreducibility of the $\widetilde{H}^{0}$--module $W_e$ it now follows that
\begin{itemize}
\item either $\dim H^0(M,\, W\bigotimes W^\vee_1)\,=\, 1$, or

\item $\text{rank}(W_1)\,=\, 1$
\end{itemize}
(see \cite[p. 1469, Lemma 3.2]{BK}).

We now observe that if $\dim H^0(M,\, W\bigotimes W^\vee_1)\,=\, 1$, then $W\,=\, W_1$
is stable. On the other hand, if $\text{rank}(W_1)\,=\, 1$, then
$$
W\,=\, W^{\oplus r}_1\, ,
$$
where $r$ is the rank of $W.$ This completes the proof of the proposition.
\end{proof}

\section{Orthogonal and symplectic quotient of $\text{SL}_n$}\label{sec3}

In this section we consider the wonderful compactification 
$\overline{G/H}$ of the following two symmetric spaces 
corresponding to the orthogonal and symplectic structures:

The first one corresponds to the involution $\sigma$
of $G\,=\, \text{PSL}(n, {\mathbb C})$ induced by the automorphism
$$
A\, \longmapsto\, (A^t)^{-1}\, 
$$
of $\text{SL}(n, {\mathbb C}), \, \, n\neq 2,4. $
The connected component of $H\,=\, G^\sigma$ is the projective orthogonal group 
$\text{PSO}(n, {\mathbb C}).$

The second one corresponds to the involution $\sigma$ of $\text{PSL}(2m,
{\mathbb C})$ induced by the automorphism
$$
A\, \longmapsto\, J^{-1}(A^t)^{-1}J \, 
$$
of $\text{SL}(2m, {\mathbb C}),$ where 
\begin{equation}\label{j}
J\, :=\, \begin{pmatrix}
0 & I_{m\times m}\\
-I_{m\times m} & 0
\end{pmatrix}.
\end{equation}
In this case, we have $H\,:=\, G^\sigma\,=\, \text{PSp}(2m, {\mathbb C}).$

\begin{theorem}\label{thm1}
For the above two cases, the tangent bundle of $\overline{G/H}$ is stable with respect
to any polarization on $\overline{G/H}.$
\end{theorem}

\begin{proof}
The Lie algebras of $G$ and $H$ will be denoted by $\mathfrak g$ and $\mathfrak h$
respectively. Consider the natural action of $H$ on ${\mathfrak g}/{\mathfrak h}.$
We will show that the $H$--module ${\mathfrak g}/{\mathfrak h}$ is irreducible.

First consider the case corresponding to the symplectic structure. In this case,
$G\,=\, \text{PSL}(2m, {\mathbb C})$ and $H\,=\,\text{PSp}(2m, {\mathbb C})$. Let
$$
\omega\, \in\, \bigwedge\nolimits^2 {\mathbb C}^{2m}
$$
be the standard symplectic form given by the matrix $J$ in \eqref{j}.
Using the symplectic form $\omega$, we identify
$\text{End}({\mathbb C}^{2m})\,=\, {\mathbb C}^{2m}\bigotimes ({\mathbb C}^{2m})^\vee$
with ${\mathbb C}^{2m}\bigotimes{\mathbb C}^{2m}$. Note that this decomposes as
$$
{\mathbb C}^{2m}\bigotimes{\mathbb C}^{2m}\,=\, \text{Sym}^2({\mathbb C}^{2m})\bigoplus
\bigwedge\nolimits^2 {\mathbb C}^{2m}\, .
$$
The $\text{PSp}(2m, {\mathbb C})$--module
${\mathfrak g}/{\mathfrak h}$ is isomorphic to the $\text{PSp}(2m, {\mathbb C})$--module
$(\bigwedge\nolimits^2 {\mathbb C}^{2m})/{\mathbb C}\cdot\omega$. It is known that the
$\text{PSp}(2m, {\mathbb C})$--module
$(\bigwedge\nolimits^2 {\mathbb C}^{2m})/{\mathbb C}\cdot\omega$ is irreducible
\cite[p. 260, Theorem 17.5]{FH} (from \cite[Theorem 17.5]{FH} it follows immediately that
the $\text{PSp}(2m, {\mathbb C})$--module $\bigwedge\nolimits^2 {\mathbb C}^{2m}$ is the
direct sum of a trivial $\text{PSp}(2m, {\mathbb C})$--module of dimension one and an
irreducible $\text{PSp}(2m, {\mathbb C})$--module).

Next consider the case corresponding to the orthogonal structure. So
$G\,=\, \text{PSL}(n, {\mathbb C})$ and $\text{PO}(n, {\mathbb C})$ is the connected
component of $H$ containing the identity element. Using the standard orthogonal form
$$
\omega'\, \in\,\text{Sym}^2({\mathbb C}^{n})
$$
on ${\mathbb C}^n$, identify ${\mathbb C}^{n}\bigotimes ({\mathbb C}^{n})^\vee$
with
$$
{\mathbb C}^{n}\bigotimes{\mathbb C}^{n}\,=\,
\text{Sym}^2({\mathbb C}^{n})\bigoplus
\bigwedge\nolimits^2 {\mathbb C}^{n}\, .
$$
Now the $H$--module ${\mathfrak g}/{\mathfrak h}$ is isomorphic to the $H$--module
$\text{Sym}^2({\mathbb C}^{n})/{\mathbb C}\cdot\omega'$. It is known that the $H$--module
$\text{Sym}^2({\mathbb C}^{n})/{\mathbb C}\cdot\omega'$ is irreducible \cite[p. 296,
Ex. 19.21]{FH} (from \cite[p. 296, Ex. 19.21]{FH} it follows that the $H$--module
$\text{Sym}^2({\mathbb C}^{n})$ is the
direct sum of a trivial $H$--module of dimension one and an irreducible $H$--module).

Fix a polarization on $\overline{G/H}$. Let $r$ be the dimension of $\overline{G/H}$.

The action of $G$ on $M$ gives an action of the isotropy subgroup $H$ on the tangent
space $T_e\overline{G/H}$. 
We note that the $H$--module $T_e\overline{G/H}$ is isomorphic to the
$H$--module ${\mathfrak g}/\mathfrak h$. Since the $H$--module ${\mathfrak g}/\mathfrak h$
is irreducible, from Proposition \ref{prop2} and Remark \ref{rem0}
we conclude that either the tangent
bundle $T\overline{G/H}$ is stable or $T\overline{G/H}$ is isomorphic to $L^{\oplus r}$
for some line bundle $L$ on $\overline{G/H}$.

Now using an argument in \cite{BK} it can be shown that $T\overline{G/H}$ is
not of the form $L^{\oplus r}$. Nevertheless, we reproduce the argument below in
order to be self--contained.

Assume that $T\overline{G/H}$ is isomorphic to $L^{\oplus r}$.
The variety $\overline{G/H}$ is unirational, because $G$ is so.
Hence $\overline{G/H}$ is simply
connected \cite[p. 483, Proposition 1]{Se}. As $T\overline{G/H}$
holomorphically splits into a direct sum of line bundles and $\overline{G/H}$
is simply connected, it follows that
$$
\overline{G/H}\,=\, ({\mathbb C}{\mathbb P}^1)^r
$$
\cite[p. 242, Theorem 1.2]{BPT}. But the tangent bundle of $({\mathbb C}{\mathbb
P}^1)^r$ is not of the form $L^{\oplus r}$. Therefore, $T\overline{G/H}$ is not
of the form $L^{\oplus r}$. This completes the proof.
\end{proof}

\begin{remark}\label{rem3}
{\rm The wonderful compactification of 
${\rm PSL}(2, \mathbb{C})/{\rm NPSO}(2, \mathbb{C})$ is isomorphic to $\mathbb{P}^2.$
The tangent bundle of $\mathbb{P}^2$ is known to be stable (see \cite{PW})}.
\end{remark}

We thank the referee for pointing out the following:

\begin{remark}\label{rem1}
{\rm The proof of Theorem \ref{thm1} remains valid if the three conditions stated in
the beginning of Section \ref{se3} are valid and the $H$--module ${\mathfrak g}/\mathfrak
h$ is irreducible. Therefore, the tangent bundle of $\overline{G/H}$ is stable with
respect to any polarization on $\overline{G/H}$ if the $H$--module ${\mathfrak g}/
\mathfrak h$ is irreducible. If $G/H$ is a non--Hermitian symmetric space, then the
$H$--module ${\mathfrak g}/\mathfrak h$ is irreducible.}
\end{remark}

\begin{remark}\label{rem4}
{\rm In Remark \ref{rem1} the hypothesis of simplicity of $\widetilde{H}^0$ is
necessary. For example, in the case of wonderful compactification of 
${\rm PSL}(4, \mathbb{C})/{\rm NPSO}(4, \mathbb{C})$ we can not use the arguments
at the end of the proof of Proposition \ref{prop2}. Though the $H$--module ${\mathfrak g}/
\mathfrak h \,\simeq\, 
sl(2,\mathbb{C})\bigotimes sl(2,\mathbb{C})$ is irreducible, it is a
tensor product of two non-trivial irreducible representations, where $sl(2,\mathbb{C})$
is the Lie algebra of ${\rm PSL}(2,\mathbb{C}).$ For the
identification of ${\rm PSO}(4, \mathbb{C})$ with ${\rm PSL}(2,\mathbb{C})\times
{\rm PSL}(2,\mathbb{C})$ (see \cite[p. 369]{FH}).}
\end{remark}

\section*{Acknowledgements}
We are very grateful to the referee for comments to improve the
exposition. In particular, Remark \ref{rem1} is due to the referee.
The first-named author acknowledges support of the J. C. Bose Fellowship.

\end{document}